\documentclass[12pt]{amsart}

\usepackage[utf8]{inputenc}

\usepackage{amssymb, graphicx}

\usepackage{amsfonts}
\usepackage{amsmath}
\usepackage[usenames,dvipsnames]{xcolor}
\usepackage[pagebackref,colorlinks,linkcolor=BrickRed,citecolor=OliveGreen,urlcolor=black,hypertexnames=true]{hyperref}

\usepackage[margin=3cm]{geometry}
\allowdisplaybreaks[1]

\newtheorem{theorem}{Theorem}[section]
\newtheorem{proposition}[theorem]{Proposition}
\newtheorem{lemma}[theorem]{Lemma}
\newtheorem{corollary}[theorem]{Corollary}
\theoremstyle{definition}

\newtheorem{definition}[theorem]{Definition}            
\newtheorem{example}[theorem]{Example}   
\numberwithin{equation}{section}

\makeatletter
\newcommand{\mycontentsbox}{%
	
	{\centerline{NOT FOR PUBLICATION}
		\addtolength{\parskip}{-2.3pt}
		\tableofcontents}}
\def\enddoc@text{\ifx\@empty\@translators \else\@settranslators\fi
	\ifx\@empty\addresses \else\@setaddresses\fi
	\newpage\mycontentsbox\newpage\printindex}
\makeatother

\begin{document}
\title[Skolem-Noether algebras]{Skolem-Noether algebras}

\author[M. Bre\v sar]{Matej Bre\v sar${}^1$}
\address{Matej Bre\v sar, Faculty of Mathematics and Physics, University of Ljubljana; 
Faculty of Natural Sciences and Mathematics, University of Maribor}
\email{matej.bresar@fmf.uni-lj.si}
\thanks{${}^1$Supported by the Slovenian Research Agency grant P1-0288.}

\author[C. Hanselka]{Christoph Hanselka${}^2$}
\address{Christoph Hanselka, Department of Mathematics, University of Auckland}
\email{c.hanselka@auckland.ac.nz}
\thanks{${}^2$Supported by the Faculty Research Development Fund (FRDF) of The University of Auckland
(project no. 3709120).}

\author[I. Klep]{Igor Klep${}^3$}
\address{Igor Klep, Department of Mathematics, University of Auckland}
\email{igor.klep@auckland.ac.nz}
\thanks{${}^3$Supported by the Marsden Fund Council of the Royal Society of New Zealand. Partially supported
by the Slovenian Research Agency grants P1-0222, L1-6722, J1-8132.}

\author[J. Vol\v{c}i\v{c}]{Jurij Vol\v{c}i\v{c}${}^4$}
\address{Jurij Vol\v{c}i\v{c}, Department of Mathematics, University of Auckland}
\email{jurij.volcic@auckland.ac.nz}
\thanks{${}^4$Supported by the University of Auckland Doctoral Scholarship.}

\date{\today}


\begin{abstract} 
An algebra $S$ is called a Skolem-Noether algebra (SN algebra for short)
if for every central simple algebra $R$, every homomorphism $R\to R\otimes S$ 
extends to an inner automorphism of $R\otimes S$. One of the important properties of such an  algebra is that each automorphism of a matrix algebra over $S$ is the composition of
an inner automorphism with an automorphism of $S$. The bulk of the paper is devoted to finding properties and examples of SN algebras. 
The classical Skolem-Noether theorem implies that every central simple algebra  is SN.
In this article it is shown that actually so is every semilocal, and hence every finite-dimensional algebra. Not every domain is SN, but, for instance,  unique factorization domains, polynomial algebras and free algebras are. Further, 
an algebra $S$ is
SN if and only if the power series algebra $S[[\xi]]$ is SN.
\end{abstract}

\maketitle
\section{Introduction}

Our main motivation for this work is the
celebrated Skolem-Noether theorem. We will state its version as given, for example, in \cite{Her}. 
But first, a word on conventions. All our algebras are assumed to be unital algebras over a  fixed field $F$, subalgebras  are assumed to contain the same unity, and all homomorphisms send $1$ to $1$.

\begin{theorem} {\bf (Skolem-Noether)} Let $A$ be simple artinian algebra with center $F$. If $R$ is a finite-dimensional simple $F$-subalgebra of $A$ and $\varphi$ is an $F$-algebra homomorphism from $R$ into $A$, then there exists an invertible element $c\in A$ such that $\varphi(x)=cxc^{-1}$ for all $x\in R$. (In other words, $\varphi$ can be extended to an inner automorphism of $A$.)
\end{theorem}
 
Recall that 
an algebra is said to be {\bf central} if its center consists of 
scalar multiples of unity.
As usual, we will use the term {\bf central simple algebra} for an algebra that is central, simple, and also finite-dimensional. 

\begin{definition}\label{def}
An algebra $S$ is a {\bf Skolem-Noether algebra} ({\bf SN algebra} for short) if for every central simple algebra $R$ and every homomorphism $\varphi:R\to R\otimes S$ there exists an invertible  element $c\in R\otimes S$ such that $\varphi(x)= cxc^{-1}$  for every $x\in R$. (Here, $R$ is identified with $R\otimes 1$).
\end{definition}

The  Skolem-Noether theorem, together 
with the well-known fact that the class  of central simple algebras is closed under tensor products, implies that every central simple algebra $S$ is an  SN algebra.  A partial converse is also true: the assertion that  central simple algebras are SN algebras implies 
an important special case of the Skolem-Noether theorem where $A$ is a central simple algebra and $R$ is its central simple subalgebra. This is because, under these assumptions,
$A$  is isomorphic to $R\otimes S$ where $S$ is also a central simple subalgebra of $A$ \cite[Corollary 4.49]{INCA}.

SN algebras naturally arise from the problem of understanding automorphism groups of tensor products of algebras. 
Unlike the case of derivations on tensor products \cite{Bre}, the general solution to this problem seems far out of reach.
For instance, while automorphisms of univariate and bivariate polynomial algebras
are well understood \cite{Jun}, already the trivariate case is wild \cite{SU}.
In another direction, functional analysts consider the question when the flip automorphism $A\otimes A\to A\otimes A$ is (approximately) inner for operator algebras $A$, see \cite{S,ER,Izu}. In this paper we settle the following special case of the above problem. If $S$ is an SN algebra and $R$ is a central simple algebra, then automorphisms of $R\otimes S$ are just compositions of inner automorphisms and automorphisms of $S$; see Proposition \ref{prop:aut}. While 
the class of SN algebras looks restrictive, our main results show that various classical and important families of algebras satisfy the SN property, for example semilocal (in particular artinian and finite-dimensional) algebras, unique factorization domains, free algebras, etc.

Some of the readers might be interested only in the case where $R=M_n(F)$,  the algebra of $n\times n$ matrices with entries in $F$.  Let us therefore mention that since $M_n(F)\otimes S$ can be identified with $M_n(S)$, the condition that $S$ is an SN algebra implies that every homomorphism from $M_n(F)$ into $M_n(S)$ can be extended to an inner automorphism of $M_n(S)$. Moreover, we show in Proposition \ref{p:mat} that the latter condition implies the SN property. 
However, this does not lead to any simplifications of our proofs, so  we persist with  central simple algebras as in Definition \ref{def}.

\subsection*{Main results and guide to the paper}
The short Section \ref{s:prelim} on preliminaries
includes Proposition \ref{p:mat}:  $S$ is an SN algebra if and only if all homomorphisms $M_n(F)\to M_n(S)$ extend to inner automorphisms. Section \ref{s:auto} 
positions SN algebras  into a wider context of automorphisms of tensor products. For instance, Proposition \ref{prop:aut} proves that given an 
SN algebra $S$ and a central simple algebra 
$R$, every automorphism of $R\otimes S$ is the composition of an inner automorphism and an automorphism of $S$. In particular, this applies to matrix algebras over SN algebras. 

We then identify classes of algebras which satisfy the SN property. In Section \ref{s:basic} we derive Lemma \ref{l}, which is the main technical tool for proving subsequent results. Section \ref{s:semiloc} culminates in Theorem \ref{thm:semiLoc} showing that semilocal algebras are SN. Hence all artinian algebras and thus all finite-dimensional algebras are SN. Section \ref{s:findim}
refines the latter result. Namely, 
every homomorphism from a central simple subalgebra $R$ of a finite-dimensional algebra $A$
into $A$ extends to an  inner
automorphism of $A$ (see Theorem \ref{tfd}). In Section \ref{s:dom} we give examples of 
domains which are SN algebras, such as 
unique factorization domains (UFDs) and free algebras, see Corollary \ref{ufd} 
and Corollary \ref{cor:free}.
Section~\ref{s:matrixpoly} uses the Quillen-Suslin theorem to prove that matrix algebras over polynomial algebras are SN.
The paper concludes with Section \ref{s:poly},
where we show that an algebra $S$ is SN if and only if the formal power series algebra $S[[\xi]]$ is SN.

\section{Preliminaries}\label{s:prelim}

The purpose of this   section is to introduce the notation and terminology, and 
prove a proposition that yields a characterization of SN algebras.

Let $R$ be a central simple algebra. Given  $w, z \in R$, we define 
the left and right multiplication operators
 $L_w, R_z:R\to R$
by  
 $$L_w(x)=wx\quad\mbox{and}\quad R_z(x)=xz.$$ 
As is well-known, every linear map from $R$ into $R$ can be written as a sum of maps of the form $L_w R_z$, $w,z\in R$  \cite[Lemma 1.25]{INCA}. Accordingly, given a basis $\{r_1,\dots,r_d\}$ of $R$, there exists $w_j,z_j\in R$ such that
$h=\sum_j L_{w_j}R_{z_j}$ satisfies 
$h(r_1)=1$ and $h(r_k)=0$, $k\ne 1$. That is, 
\begin{equation}\label{oh}
\sum_j w_jr_1z_j = 1\quad \mbox{and}\quad \sum_j w_jr_kz_j = 0\,\,\,\mbox{if $k> 1$.}
\end{equation}

We will be mostly concerned with  tensor product algebras $R\otimes S$. 
Here $R,S$ are algebras over a field $F$ and the tensor product is taken over $F$.
As usual,  
 we identify  $R$ by $R\otimes 1$, and, accordingly, often write $r\otimes 1\in R\otimes 1$ simply as $r$.
 Let us  point out an elementary fact that will be used freely  without further reference. If the $r_i$'s are linearly independent elements 
in $R$, then for all  $p_j\in R$ and  $s_j, t_i\in S$,
\begin{equation}\label{nov}
\sum_i r_i\otimes t_i = \sum_j  p_j\otimes s_j
\end{equation}
implies that each $t_i$ lies in the linear span of the $s_j$'s \cite[Lemma 4.9]{INCA}. Similarly,  assuming  that the $t_i$'s are linearly independent, it follows from \eqref{nov}  that each $r_i$ lies in the linear span of the $p_j$'s.

By rad$(S)$ we denote the {\bf Jacobson radical} of the algebra $S$. Recall that $S$ is called a {\bf semilocal algebra}
if $S/{\rm rad}(S)$ is a semisimple algebra, i.e., isomorphic to a finite direct product of simple artinian algebras. 
In the special case where  $S/{\rm rad}(S)$ is a division algebra, $S$ is called a {\bf local algebra}. Finally, we say that  $S$ is a {\bf stably finite algebra} if for all $n\ge 1$ and all $x,y\in M_n(S)$, $xy=1$ implies $yx=1$.

To conclude the section we  give an alternative characterization of the SN property.
In order to show that $S$ is an SN algebra it suffices to 
verify 
the condition of Definition \ref{def} for $R=M_n(F)$, i.e., 
all $F$-algebra homomorphisms $M_n(F)\to M_n(S)$ 
are given by conjugation.

\begin{proposition}\label{p:mat}
Let $S$ be an algebra and suppose that for every $n\in\mathbb{N}$ and a homomorphism $\varphi:M_n(F)\to M_n(S)$ there exists $c\in M_n(S)$ such that $\varphi(x)=cxc^{-1}$ for every $x\in R$. Then $S$ is an SN algebra.
\end{proposition}

\begin{proof}
Let $R$ be a central simple algebra and $\varphi:R\to R\otimes S$ a homomorphism. Let $e\in\mathbb{N}$ be the exponent of $R$, i.e., the order of $R$ as an element of the Brauer group of $F$ \cite[Definition 4.5.12]{GS}. Let
$$\tilde{\varphi}={\rm id}^{e-1}\otimes \varphi: R^{\otimes e}\to R^{\otimes e}\otimes S.$$
Since
$$R^{\otimes e}\cong M_{(\deg R)^e}(F),$$
by assumption there exists $c\in R^{\otimes e}\otimes S$ such that
$\tilde{\varphi}(x)=cxc^{-1}$ for every $x\in R^{\otimes e}$.

If $e>1$, we can write $c$ as
$$c=\sum_{i_2,\dots,i_e,j} a_{i_2,\dots,i_e,j}\otimes r_{i_2}\otimes\cdots \otimes r_{i_e}\otimes s_j$$
for some $a_{i_2,\dots,i_e,j},r_i\in R$ and $s_j\in S$ where $\{r_i\}_i\subset R$ and $\{s_j\}_j\subset S$ are linearly independent sets. If $x=x_1\otimes 1\otimes\cdots\otimes 1$ for $x_1\in R$, then $\tilde{\varphi}(x)c-cx=0$ becomes
$$ \sum_{i_2,\dots,i_e,j}
(xa_{i_2,\dots,i_e,j}-a_{i_2,\dots,i_e,j}x)\otimes r_{i_2}\otimes\cdots r_{i_e}\otimes s_j=0.$$
Since the elements $r_{i_2}\otimes\cdots r_{i_e}\otimes s_j$ form a linearly independent set in $R^{\otimes (e-1)}\otimes S$, we conclude that $xa_{i_2,\dots,i_e,j}=a_{i_2,\dots,i_e,j}x$ for all $a_{i_2,\dots,i_e,j}\in R$ and $x\in R$. As $R$ is central we have $a_{i_2,\dots,i_e,j}\in F$ and therefore $c\in 1\otimes R^{\otimes (e-1)}\otimes S\cong R^{\otimes (e-1)}\otimes S$. Consequently the homomorphism
$$\hat{\varphi}={\rm id}^{e-2}\otimes \varphi: R^{\otimes (e-1)}\to R^{\otimes (e-1)}\otimes S$$
satisfies $\hat{\varphi}(x)=cxc^{-1}$ for every $x\in R^{\otimes (e-1)}$. Continuing by induction we conclude that $c\in R\otimes S$ and $\varphi(x)=cxc^{-1}$ for all $x\in R$, so $S$ is an SN algebra.
\end{proof}

While Proposition \ref{p:mat} seemingly facilitates demonstrating that $S$ is an SN algebra, it does not  simplify our proofs in the sequel.

\section{SN algebras and automorphisms}\label{s:auto}

In this section we give a few motivating
results and prove that every automorphism of 
a matrix algebra over an SN algebra $S$
is an inner automorphism composed
with an automorphism of $S$, see Corollary \ref{cor:autMn}.

We begin with a  proposition which justifies the requirement in  Definition \ref{def} that the algebra  $R$ 
is central simple.

\begin{proposition}\label{rjes}
Let $R$ be a subalgebra of an algebra $S$. If the homomorphism $x\otimes 1\mapsto 1\otimes x$ from $R=R\otimes 1$ 
into $R\otimes S$ can be extended to an inner automorphism of $R\otimes S$, then 
 $R$ is a central simple algebra.
\end{proposition}

\begin{proof}
 By assumption, there exists an invertible element $a\in R\otimes S$ such that 
$$1\otimes x =a(x\otimes 1)a^{-1}$$ for all $x\in R$.
 Let us write $a=\sum_{i=1}^m u_i\otimes v_i$ and $a^{-1}= \sum_{j=1}^n w_j\otimes z_j$.
Accordingly,
\begin{align}\label{k}
1\otimes x   =& \Big(\sum_{i=1}^m u_i\otimes v_i\Big)(x\otimes 1)\Big( \sum_{j=1}^n w_j\otimes z_j\Big) \\
=&\sum_{i=1}^m \sum_{j=1}^n u_ixw_j\otimes v_iz_j.\nonumber
\end{align}
This implies that every $x\in R$ lies in the linear span of all $v_iz_j$, $i=1,\dots,m$, $j=1,\dots,n$. Thus, $R$ is finite-dimensional.
On the other hand,  \eqref{k} implies that for every nonzero  $x\in R$, $1$ lies in $RxR$. This means that $R$ is simple. 
Finally, if $z$ lies in the center of $R$, then $1\otimes z = a(z\otimes 1)a^{-1} = z\otimes 1$, which readily implies that $z$ is a scalar multiple of $1$,  as desired. 
\end{proof}

The question of when the automorphism $x\otimes y\mapsto y\otimes x$ of $R\otimes R$ is inner was initiated by Sakai \cite{S}  in the C$^*$-algebra context, and investigated further by Bunce \cite{Bu}. The following corollary is an extension of \cite[Theorem 2]{Bu}.

\begin{corollary}\label{rjes2}
Let $R$ be an arbitrary algebra. The homomorphism $x\otimes 1\mapsto 1\otimes x$ from $R=R\otimes 1$ into $R\otimes R$ can be extended to an inner automorphism of $R\otimes R$ if and only if  
 $R$ is a  central simple algebra.
\end{corollary}

\begin{proof}
If $R$ is a central simple algebra, then so is $R\otimes R$ \cite[Corollary 4.44]{INCA}, and so every homomorphism from $R$ into $R\otimes R$ can be extended to an inner automorphism by the Skolem-Noether theorem. The converse follows from Proposition \ref{rjes}.
 \end{proof}

The next proposition yields another motivation for considering SN algebras.

\begin{proposition}\label{prop:aut}
Let $R$ be a central simple algebra and let $S$ be an SN algebra. Then every automorphism $\varphi$ of $R\otimes S$ is the composition of an inner automorphism and an automorphism of the form ${\rm id}_R\otimes \sigma$ where $\sigma$ is an automorphism of $S$.
\end{proposition}

\begin{proof}
By assumption, the restriction of $\varphi$ to $R$ can be extended to an inner automorphism $x\mapsto cxc^{-1}$ of $R\otimes S$. Considering the automorphism $x\mapsto 
c^{-1}\varphi(x)c$ we thus see that there is no loss of generality in assuming that $\varphi$ acts as the identity on $R$. Note that the proposition will be proved by showing 
that $\varphi$ 
maps $1\otimes S$ into itself. Pick $s\in S$. We can write $\varphi(1\otimes s)$ as
$\sum_{j} p_j\otimes s_j$
where the $s_j$'s are linearly independent. 
Since $1\otimes s$ commutes with $x\otimes 1$ for every $x\in R$ it follows that so does $\varphi(1\otimes s)$. This implies that
$$\sum_j (p_jx-xp_j)\otimes s_j =0.$$ 
As the $s_j$'s are linearly independent it follows that $p_jx-xp_j=0$ for each $j$ and each $x\in R$.
Hence, since $R$ is central,  each $p_j$ is a scalar multiple of $1$. Consequently, $\varphi(1\otimes s)\in 1\otimes S$.
\end{proof}

If $R=M_n(F)$, then
 $R\otimes S$ can be identified with $M_n(S)$, and the proposition gets the following form.

\begin{corollary}\label{cor:autMn}
If $S$ is an SN algebra, then every automorphism $\varphi$ of $M_n(S)$ is of the form $$\varphi\big((s_{ij})\big) = c\big(\sigma(s_{ij})\big)c^{-1}$$ where $c$ is an invertible element in $M_n(S)$ and $\sigma$ is an automorphism of $S$.
\end{corollary}

This result is known in the case where $S$ is either an artinian algebra \cite[Theorem 3.13]{BO}, a UFD \cite[Corollary 15]{I}, or a commutative local algebra (see, e.g., \cite[p. 163]{K}). As we will see,  all these algebras are SN algebras. On the other hand, \cite{I} 
shows that the commutative domain $\mathbb Z[\sqrt{-5}]$ does not satisfy the conclusion of Corollary \ref{cor:autMn}.
We give an algebra with the same property in Example \ref{ex:elliptic}.

\section{Basic lemma}\label{s:basic}

All our main results will be derived from the following technical lemma. Its proof  will use some ideas from the proof of the (special case of)   Skolem-Noether theorem  given in \cite[pp. 13--14]{INCA}.

\begin{lemma}\label{l}
Let $R$ be a  central simple algebra with basis $\{r_1,\dots,r_d\}$ and let $S$ be an arbitrary algebra. Then  $\varphi:R\to R\otimes S$ is a homomorphism 
if and only if there exist   $c_1,\dots,c_d\in R\otimes S$ such that
\begin{enumerate}
\item[(a)] $\varphi(x)=\sum_{k=1}^d c_k x r_{k}$ for all $x\in R$,
\item[(b)] $\varphi(x)c_k=c_kx$ for all $x\in R$, and
\item[(c)] $\sum_{k=1}^d c_kr_k = 1.$
\end{enumerate}
Moreover, writing $c_k= \sum_{l=1}^{d} r_{l}\otimes s_{kl}$, we 
have that for each $k$ and $l$ there exists $b_{kl}\in R\otimes S$ such that $$b_{kl}c_k = 1\otimes s_{kl}.$$ Accordingly, if $S$ is stably finite and there exist $k$ and $l$ such that  $s_{kl}$ is invertible in $S$, then $c=c_k$ is invertible in $R\otimes S$ and $$\varphi(x)=cxc^{-1}$$ for all $x\in R$. 
\end{lemma}

\begin{proof}
Since $R$ is finite-dimensional, there exist finitely many $s_i\in S$ and 
linear maps $f_i:R\to R$
 such that
$$\varphi(x)= \sum_i f_i(x)\otimes s_i
$$
for all $x\in S$. By  \cite[Lemma 1.25]{INCA}  there exist $w_{ij}, z_{ij}\in R$ such that 
$$f_i = \sum_{j} L_{w_{ij}}R_{z_{ij}}.$$
 Consequently, for every $x\in R$ we have
\begin{align*}
\varphi(x) &=\sum_i \Big( \sum_{j} w_{ij}xz_{ij}\Big)\otimes s_i\\
&=\sum_i \sum_{j} (w_{ij}\otimes s_i)x z_{ij}.
\end{align*}
Writing each $z_{ij}$ as a linear combination of $r_1,\dots,r_d$ we see that $\varphi$ is of the form described in (a).

 We now use the multiplicativity of  $\varphi$, i.e., 
 $\varphi(xy)= \varphi(x)\varphi(y)$ for all $x,y\in R$. In view of (a) we can rewrite this as 
 \begin{equation} \label{nm} \sum_{k=1}^d c_kxyr_k = \sum_{k=1}^d  \varphi(x) c_k yr_k. \end{equation}
 Pick  $w_j,z_j\in R$  such that \eqref{oh} holds. 
Setting $y= w_j$ in  \eqref{nm}, 
multiplying  the identity, so obtained, from the right by $z_j$,  and then summing up over all $j$ we get
 $$
 \sum_j  \sum_{k=1}^d c_kx w_j r_k z_j = \sum_j  \sum_{k=1}^d \varphi(x) c_k w_j r_k z_j,
 $$
 that is,
  $$
 \sum_{k=1}^d c_kx \Big(\sum_j w_j r_k z_j\Big) =  \sum_{k=1}^d \varphi(x) c_k \Big(\sum_j w_j r_k z_j\Big).
 $$
 By \eqref{oh} this reduces to $c_1x = \varphi(x)c_1$. Of course,  the same proof applies  to every $c_k$, so (b) holds. Finally, (c) follows 
from $\varphi(1)=1$.

A direct verification shows that  (a), (b), and (c) imply that $\varphi$ is a homomorphism.

Let us write $c_1 = \sum_l r_{l}\otimes s_{1l}$, and let $w_j,z_j$ be   as above. Using (b) we obtain
\begin{align*}
\sum_j w_j\varphi(z_j) c_1 &= \sum_j w_jc_1z_j\\
&= \sum_j\sum_l w_j ( r_{l}\otimes s_{1l})z_j\\
&=\sum_l\Big(\sum_j w_jr_{l}z_j\Big)\otimes s_{1l}\\
&= 1\otimes s_{11}.
\end{align*}
Thus, $b_{11}= \sum_j w_j\varphi(z_j)$ satisfies $b_{11}c_1 = 1\otimes s_{11}$. Similarly we find other $b_{kl}$'s. 

Finally, assume that $s_{kl}$ is invertible in $S$ for some  $k$ and $l$. Then 
$(1\otimes s_{kl}^{-1})b_{kl}$ is a left inverse of $c_k$. If $S$ is stably finite, then the result by Montgomery  \cite[Theorem 1]{Mon} implies that
this element is also a right inverse. Therefore, (b) shows that $c=c_k$ satisfies  $\varphi(x)=cxc^{-1}$ for all $x\in R$. 
\end{proof}

We continue with a simple application
of Lemma \ref{l}, showing that local algebras
are SN. This result will be generalized 
to semilocal algebras (with a considerably more involved proof)
in the next section, see Theorem \ref{thm:semiLoc}.

\begin{corollary}\label{t1}
Every local algebra  is an SN algebra.
\end{corollary}

\begin{proof}
Let $r_k,s_{kl}$ be elements from Lemma \ref{l}. From (c) it follows that  $$\sum_{k,l} r_{l}r_k\otimes s_{kl} = 1\otimes 1.$$ 
This implies that
$1$ lies in the linear span of $s_{kl}$. Consequently, at least one $s_{kl}$  does not lie in rad$(S)$. Since $S$ is local it follows that $s_{kl}$  is invertible in $S$. 
As  $S$ is stably finite  \cite[Theorem 20.13]{Lam},  the last assertion of Lemma \ref{l} shows that there exists $c\in R\otimes S$ such that 
$\varphi(x)=cxc^{-1}$ for all $x\in R$.
\end{proof}

\section{Semilocal algebras}\label{s:semiloc}

The main result of this section is Theorem
\ref{thm:semiLoc} showing that
semilocal algebras are SN. We begin with a simple lemma.

\begin{lemma}\label{prod}
If  $S_1$ and $S_2$ are SN algebras, then so is their direct product $S_1\times S_2$.
\end{lemma}

\begin{proof}
Recall that $R\otimes (S_1\times S_2)$ can be identified with $(R\otimes S_1)\times (R\otimes S_2)$. Take a homomorphism $$\varphi:R\to (R\otimes S_1)\times (R\otimes S_2).$$
Writing $$\varphi(x)=(\varphi_1(x),\varphi_2(x))$$ it is immediate that $\varphi_i$ is a homomorphism from $R$ into $R\otimes S_i$, $i=1,2$. By assumption,  there exist $c_i\in R\otimes S_i$  such that
$\varphi_i(x)= c_ixc_i^{-1}$ for all $x\in R$, $i=1,2$. Hence, $$c=(c_1,c_2)\in (R\otimes S_1)\times (R\otimes S_2)$$ satisfies $\varphi(x)=cxc^{-1}$ for all $x\in R$.
\end{proof}

As mentioned in the introduction, the Skolem-Noether theorem implies that every central simple algebra is an SN algebra. 
With a little extra effort we can extend this to  semisimple algebras.

\begin{lemma}\label{lsemi}
Every semisimple algebra is an SN algebra.
\end{lemma}

\begin{proof}
In view of Lemma \ref{prod} it suffices to consider the case where $S$ is  simple artinian. 
Let $R$ be a central simple algebra. The algebra $R\otimes S$
is then simple  \cite[Theorem 4.42]{INCA}. 
We claim that it is  also artinian. Indeed, considering $R\otimes S$ as a left $S$-module in the natural way  we see that
it is isomorphic to the  left $S$-module $S^d$ where $d$ is the dimension of $R$, and that
 a descending chain of left ideals of $R\otimes S$
is also a descending chain of left $S$-submodules. The desired conclusion thus follows from the fact that   $S^d$ is  artinian.    

Let $K$ be the center of $S$. 
The center of $R\otimes S$  is equal to $1\otimes K$ \cite[Corollary 4.32]{INCA}, which we     identify with $K$. Consider $R\otimes K$ as an algebra over $K$ in the usual way. Clearly, 
it is finite-dimensional and, again by \cite[Theorem 4.42]{INCA}, simple.

Now, given a homomorphism $\varphi:R\to R\otimes S$, we define $$\Phi:R\otimes K\to R\otimes S$$ by $$\Phi(x\otimes k)= \varphi(x)(1\otimes k).$$ Note that $\Phi$ is a $K$-algebra homomorphism. The Skolem-Noether theorem thus tells us that there exists $c\in R\otimes S$ such that
$\Phi(x\otimes k)= c(x\otimes k)c^{-1}$ for all $x\in R$, $k\in K$. Setting $k=1$ we get the desired conclusion.
\end{proof}

Our goal is to show that semilocal algebras are SN algebras by reducing the general case to the semisimple case. We will actually prove a general reduction theorem whose possible applications are not limited to  semilocal algebras. To this end, we need the following lemma. From its nature one would expect that it is known,  but we were unable to find a good reference. 
We include a short proof for the sake of completeness.

\begin{lemma}\label{lemrad}
If $R$ is a central simple algebra, then ${\rm rad}(R\otimes S) = R\otimes {\rm rad}(S)$ for every algebra $S$.
\end{lemma}

\begin{proof}
	As an ideal of $R\otimes S$, ${\rm rad}(R\otimes S)$ is necessarily of the form 
	$R\otimes I$ for some ideal $I$ of $S$ \cite[Theorem 4.42]{INCA}.
	We will show that $I\subseteq {\rm rad}(S)$, by making use of the following  characterization of rad$(A)$:
	 $v\in {\rm rad}(A)$ if and only if
	$1- vx$ is invertible for every $x\in A$.
Take  $u\in I$. Since $1\otimes u \in {\rm rad}(R\otimes S)$ it follows that 

	$$1\otimes (1-ux)=1\otimes 1 - 1\otimes ux = 1\otimes 1 - (1\otimes u)(1\otimes x)$$
	is invertible in $R\otimes S$ for every $x\in S$. However, this is possible only if $1-ux$ is invertible, implying that $u\in {\rm rad}(S)$. Thus, $I\subseteq {\rm rad}(S)$,
	and so $ {\rm rad}(R\otimes S)\subseteq R\otimes {\rm rad}(S)$.
	
	As the lemma is well-known if $R=M_n(F)$ (see, e.g., \cite[pp. 57-58]{Lam}),
we will establish the converse inclusion  by reducing the general case to this one.
Take a splitting field $K$ for $R$ which is a finite separable extension of $F$ (see, e.g., \cite[Proposition 4.5.4]{GS}). Then  $K\otimes R$ may be identified with  $M_n(K)$ for some $n\ge 1$, and, therefore, $K\otimes R\otimes S$ may be identified with $M_n(K\otimes S)$. Thus, by what we pointed out at the beginning of the paragraph, we have
$${\rm rad}(K\otimes R\otimes S) = M_n({\rm rad}(K\otimes S)).$$ 
By \cite[Theorem 5.17]{Lam}, ${\rm rad}(K\otimes S) = K\otimes {\rm rad}(S)$, so that
\begin{equation}\label{rs}{\rm rad}(K\otimes R\otimes S) = M_n(K)\otimes {\rm rad}(S).\end{equation} 
According to  \cite[Theorem 5.14]{Lam},
$${\rm rad}(R\otimes S) = (R\otimes S)\cap {\rm rad}(K\otimes R\otimes S),$$
and hence, by \eqref{rs}, 
$${\rm rad}(R\otimes S) = (R\otimes S)\cap  (M_n(K)\otimes {\rm rad}(S)).$$
Since  both $R\otimes S$ and $M_n(K)\otimes {\rm rad}(S)$ readily contain $R\otimes {\rm rad}(S)$, 
 it follows that $R\otimes {\rm rad}(S)\subseteq {\rm rad}(R\otimes S)$. 
\end{proof}

We can now prove the announced reduction theorem. 

\begin{theorem}\label{trad} If an algebra $S$ is stably finite and  $S/{\rm rad}(S)$ is  an  SN algebra, then 
$S$  is  an  SN algebra.
\end{theorem}

\begin{proof}
Let $R$ be central simple and 
 write   $$J=R\otimes {\rm rad}(S).$$
  Take a homomorphism $\varphi:R\to R\otimes S$. We define 
	$$\Phi:R\to (R\otimes S)/J$$ by 
	$$\Phi(x)=\varphi(x)+J.$$
	Since $(R\otimes S)/J$ is canonically isomorphic to $R\otimes (S/{\rm rad}(S))$,  and $S/{\rm rad}(S)$ is  assumed to be an SN algebra, it follows that
	there exists an invertible element $a\in (R\otimes S)/J$ such that \begin{equation*}\label{bg}\Phi(x)= a(x+J)a^{-1}\,\,\mbox{ for all $x\in R$.}\end{equation*}
As $J$ is, by Lemma \ref{lemrad}, the Jacobson radical of $R\otimes S$,  it follows that there exists an invertible element $b\in R\otimes S$ such that $a = b +J$. Obviously, we have
	$$\varphi(x)-bxb^{-1}\in J\,\,\mbox{ for all $x\in R$,}$$
that is,
	$$b^{-1}\varphi(x)b -x\in J\,\,\mbox{ for all $x\in R$.}$$
Replacing the role of $\varphi$ by the homomorphism $x\mapsto b^{-1}\varphi(x)b$	we see that without loss of generality we may assume that $b=1$. Thus,
\begin{equation}\label{kk}\varphi(x) -x\in J\,\,\mbox{ for all $x\in R$.}\end{equation}
Now apply Lemma \ref{l}. Picking a basis $\{r_1,\dots,r_d\}$ of $R$, we can thus find  $s_{kl}\in S$, $k,l=1,\dots,p$, such that 
\begin{equation}\label{ll}\varphi(x)=\sum_{k=1}^d \sum_{l=1}^d r_l x r_k \otimes s_{kl} \,\,\mbox{ for all $x\in R$,}\end{equation}
and our goal is to show that at least one $s_{kl}$ is invertible in $S$.

Let $\lambda_k\in F$ be such that  $1=\sum_{k=1}^d {\lambda_k}r_k$. Then
$$x=1\cdot x\cdot 1 = \sum_{k=1}^d \sum_{l=1}^d (\lambda_k\lambda_l) r_l x r_k \,\,\mbox{ for all $x\in R$}.$$
Using \eqref{kk} and \eqref{ll} we thus obtain
\begin{equation}\label{ro}
\sum_{k=1}^d \sum_{l=1}^d r_l x r_k \otimes (s_{kl} - \lambda_k\lambda_l\cdot 1)\in J \,\,\mbox{ for all $x\in R$}.
\end{equation}
We may assume that $\lambda_1\ne 0$. Choose  $w_j,z_j\in R$ that satisfy \eqref{oh}. 
Denoting the expression in \eqref{ro} by $\rho(x)$, we have
\begin{align*}
\sum_j w_j\rho(z_j) &= \sum_j \sum_{k=1}^d \sum_{l=1}^d w_jr_l z_j r_k \otimes (s_{kl} - \lambda_k\lambda_l \cdot1)\\
&= \sum_{k=1}^d \sum_{l=1}^d \Big(\sum_j w_jr_l z_j \Big) r_k \otimes (s_{kl} - \lambda_k\lambda_l\cdot 1)\\
&= \sum_{k=1}^d   r_k \otimes (s_{k1} - \lambda_k\lambda_1 \cdot1).
\end{align*}
As $\rho$ maps into $J$ it follows that 
$$\sum_{k=1}^d   r_k \otimes (s_{k1} - \lambda_k\lambda_1 \cdot 1)\in J =R\otimes {\rm rad}(S).$$
Since the $r_k$'s are linearly independent, we must have  $s_{k1} - \lambda_k\lambda_1 \cdot 1 \in {\rm rad}(S)$
for each $k$. In particular, $s_{11} = \lambda_1^2\cdot 1 + u$ for some $u\in  {\rm rad}(S)$. Since $\lambda_1\ne 0$ it follows that $s_{11}$ is invertible, as desired. 
\end{proof}

 We are now in a position to give our main result. 
\begin{theorem}\label{thm:semiLoc}
Every semilocal algebra $S$ is an SN algebra.
\end{theorem}

\begin{proof}Since $S$  is stably finite \cite[Theorem 2.13]{Lam} and the algebra $S/{\rm rad}(S)$ is semisimple,    
the theorem follows from Lemma \ref{lsemi} and  Theorem \ref{trad}.
\end{proof}

\begin{corollary} \label{art}
Every artinian algebra is an SN algebra. 
\end{corollary}

\section{Finite-dimensional algebras}\label{s:findim}

Corollary \ref{art}  shows that every finite-dimensional algebra  is an   SN algebra.
The next result gives a strengthening of this property.

\begin{theorem}\label{tfd}
Let $A$ be a finite-dimensional algebra and let $R$ be its central simple subalgebra. Then every homomorphism from $R$ into $A$ can be extended to an inner automorphism of $A$.
\end{theorem}

\begin{proof}
Assume first that $R=M_n(F)$. Then $A$ contains a set of $n\times n$ matrix units and is therefore isomorphic to $M_n(S)\cong R\otimes S$ for some subalgebra $S$ of $A$ \cite[Lemma 2.52]{INCA}. Since $S$ is also finite-dimensional, the desired conclusion follows from Corollary \ref{art}.

Now let $R$ be an arbitrary central simple algebra. We may assume that the field $F$ is infinite, for otherwise $R\cong M_n(F)$ by   Wedderburn's theorem on finite division rings.
Let 
$\varphi$ be a homomorphism from $R$ into $A$.
Take  a splitting field $K$ for $R$. Identifying $K\otimes R$  with  $M_n(K)$, $n \ge 1$,  it follows from the preceding paragraph that there exists 
$b\in K\otimes A$ such that \begin{equation*} \label{y} ({\rm id}_K\otimes \varphi)(y) = byb^{-1}\end{equation*} for all $y\in K\otimes R$.
In particular, $$(1\otimes \varphi(x))b = b(1\otimes x)$$ for all $x\in R$.
 Writing $b=\sum_{i=1}^m k_i\otimes a_i$ with the $k_i$'s  linearly independent, it follows that 
$$\sum_{i=1}^m k_i \otimes (\varphi(x)a_i - a_ix)=0,$$
and so 
$\varphi(x)a_i = a_ix$
for all $x\in R$ and every $i$. Hence we see that it suffices to show that span$_F\{a_1,\dots,a_m\}$ contains an  element which is invertible in $A$. 

As a finite-dimensional algebra, $A$ can be considered as a subalgebra of $M_N(F)$ for some $N\ge 1$. Take the polynomial
$$f(\xi_1,\dots,\xi_m)=\det\Big(\sum_{i=1}^m \xi_i a_i\Big) \in F[\xi_1,\dots,\xi_m].$$
Note that $K\otimes A$ can be viewed as a subalgebra of $M_N(K)$. 
Since $b$ is invertible in $K\otimes A$, we know that span$_K\{a_1,\dots,a_m\}$ contains an invertible element in $K\otimes A$. This clearly implies that $f$ is a nonzero polynomial. 
As $F$ is infinite,  there exist $\lambda_i\in F$ such that $f(\lambda_1,\dots,\lambda_m)\ne 0$. That is, span$_F\{a_1,\dots,a_m\}$ contains an element $c$ which is invertible in $M_N(F)$. However,
since we are in finite dimensions, $c^{-1}$ is a polynomial in $c$. Thus, $c$ is invertible in  $A$. 
\end{proof}

Using the standard homomorphism construction we will now see that Theorem \ref{tfd} can be used for  showing that all derivations from $R$ 
into any $R$-bimodule $M$ are inner (in accordance with the conventions mentioned at the very beginning of the paper, we assume that our bimodules are unital). 
This is, of course,  a  well-known result. Another way of stating it is that central simple algebras are separable.

\begin{corollary}\label{cfd}
Every derivation from a central simple algebra $R$ into an arbitrary $R$-bimodule $M$ is inner.
\end{corollary}

\begin{proof} Let $d:R\to M$ be  a derivation.
As  a finite-dimensional subspace of $M$, $d(R)$ generates a finite-dimensional subbimodule of $M$. Therefore, there is no loss of generality in assuming that $M$ itself is finite-dimensional. 

Let $\widetilde{A}$ be the set of all matrices of the form $\left[
    \begin{smallmatrix}
      x & u\\
      0&x
    \end{smallmatrix}\right]$, where $x\in R$ and $u\in M$. Note that  $\widetilde{A}$ is a (finite-dimensional!) algebra under the standard matrix operations.
Let $\widetilde{R}$ be its subalgebra consisting of all matrices of the form $\left[
    \begin{smallmatrix}
      x & 0\\
      0&x
    \end{smallmatrix}\right]$,  $x\in R$. Of course,  $\widetilde{R}\cong R$.

Define $\varphi:\widetilde{R}\to \widetilde{A}$ by 
$$\varphi\left(\left[
    \begin{matrix}
      x & 0\\
      0&x
    \end{matrix}\right]\right)=\left[
    \begin{matrix}
      x & d(x)\\
      0&x
    \end{matrix}\right].$$
		One immediately checks  that $\varphi$ is a homomorphism. By Theorem \ref{tfd} there exists an invertible element $c= \left[
    \begin{smallmatrix}
      t & v\\
      0&t
    \end{smallmatrix}\right] \in \widetilde{A}$ such that $\varphi(\tilde{x})= c\tilde{x}c^{-1}$ for all $\tilde{x}\in\widetilde{R}$. Consequently,  $\varphi(\tilde{x})c= c\tilde{x}$, that is,
		$$\left[
    \begin{matrix}
      x & d(x)\\
      0&x
    \end{matrix}\right]\cdot \left[
    \begin{matrix}
      t & v\\
      0&t
    \end{matrix}\right] = \left[
    \begin{matrix}
      t & v\\
      0&t
    \end{matrix}\right]\cdot \left[
    \begin{matrix}
      x & 0\\
      0&x
    \end{matrix}\right]$$
		for all $x\in R$. This yields $$xt = tx\,\,\,\mbox{ and}\,\,\, xv + d(x)t = vx$$   for all $x\in R$. Since $R$ is central, the first identity shows that $t\in F$. Moreover, $t\ne 0$ for $c$ is invertible.  Hence $w=t^{-1}v$ satisfies  $d(x) = wx - xw$ by the second identity.
\end{proof}

\section{Domains}\label{s:dom}

In this section we give classes of domains which are SN algebras. For instance, 
UFDs and free algebras are SN algebras (Corollaries \ref{ufd} and \ref{cor:free}).

As the coordinate ring of an elliptic curve demonstrates (see Example \ref{ex:elliptic}),
not every commutative domain is an SN algebra.
However, the following proposition shows that 
every domain $S$ embedded into a division algebra satisfies
a certain weaker condition.

\begin{proposition}\label{p}
Let $R$ be a  central simple algebra, and let an algebra $S$ be a domain which can be embedded into a division algebra $D$. If $\varphi$ is a homomorphism from $R$ into 
 $R\otimes S$, then there exists  $c\in R\otimes S$ which is  invertible in $R\otimes D$, in fact $$c^{-1} = (1\otimes s^{-1})b\in R\otimes D$$ for some nonzero $s\in S$ and $b\in R\otimes S$,  such that
$$\varphi(x)=cxc^{-1}$$ for all $x\in R$. Moreover, if $\{r_1,\dots,r_d\}$ is a basis of $R$, $b=\sum_{k=1}^d 
r_k\otimes s_k$ for some $s_k\in S$, and
$c=\sum_{l=1}^d r_l\otimes t_l$  for some $t_l\in S$, then   $$t_ls^{-1}s_k\in S$$ for all $k$ and $l$.  
\end{proposition}

\begin{proof}
Not every $c_k$ from Lemma \ref{l} can be $0$ (in view of (c)), and so $s_{kl}\ne 0$ for some $k$ and $l$.
Set $c=c_k$ and $s=s_{kl}$. By the lemma we have $\varphi(x)c=cx$ for all $x\in R$ and $bc=1\otimes s$ for some 
$b\in R\otimes S$. Of course,  $s$ is invertible in $D$. Therefore $(1\otimes s^{-1})b$ is a left inverse of $c$ in $R\otimes D$. By \cite[Theorem 1]{Mon}, a left inverse in $R\otimes D$ is also a right inverse, so $c^{-1}=(1\otimes s^{-1})b$.

Now take a basis  $\{r_1,\dots,r_d\}$ of $R$, and let us write
$b=\sum_{k=1}^d r_k\otimes s_k$  and 
$c=\sum_{l=1}^d r_l\otimes t_l$. Then 
$$\varphi(x)= cxc^{-1}=\sum_{k=1}^d\sum_{l=1}^d r_lxr_k\otimes t_ls^{-1}s_k.$$
for all $x\in R$. Pick $w_j,z_j\in R$ satisfying \eqref{oh}.
 We have
\begin{align*}
\sum_j w_j\varphi(z_j) &= \sum_j \sum_{k=1}^d \sum_{l=1}^d w_jr_l z_j r_k \otimes t_ls^{-1}s_k\\
&=  \sum_{k=1}^d \sum_{l=1}^d \Big( \sum_j w_jr_l z_j \Big)r_k \otimes t_ls^{-1}s_k\\
&=  \sum_{k=1}^d r_k \otimes t_1s^{-1}s_k.
\end{align*}
Since the left hand side, i.e.  $\sum_j w_j\varphi(z_j)$, lies in $R\otimes S$, so does the right hand side. This readily yields that
$t_1s^{-1}s_k\in S$.
\end{proof}

\begin{corollary}\label{ufd}
Every UFD is an SN algebra.
\end{corollary}

\begin{proof}
Let $S$ be a UFD and let $R$, $\varphi$, $b$, $c$, $s$, $s_k$, $t_l$ 
be as in Proposition \ref{p}. Since $S$ is a UFD, $t_1,\dots,t_d$ have a greatest common divisor $e$ and $c$ can be replaced with $(1\otimes e^{-1})c$, we can without loss of generality assume that $t_1,\dots,t_d$ are coprime. Then it suffices to prove that $s^{-1}s_k\in S$ for every $k$. Since $t_ls^{-1}s_k\in S$ for every $k,l$, we see that $s$ divides $t_ls_k$ for every $l,k$. Let $p$ be a prime such that $p^n$ divides $s$. Suppose that $p^n$ does not divide $s_{k_0}$ for some $k_0$. Since $p^n$ divides $t_ls_{k_0}$ for every $l$, we conclude that $p$ divides $t_l$ for every $l$, which contradicts the assumption about $t_l$ being coprime. Hence $s$ divides $s_k$ for every $k$.
\end{proof}

We now move to the noncommutative setting. Since embeddings of noncommutative domains into division rings can be ill-behaved or nonexistent, one needs stronger assumptions than in Corollary \ref{ufd}. Let $S$ be an arbitrary ring. The {\bf inner rank} of $A\in S^{m\times n}$ is the least $r$ such that $A=BC$ for some $B\in S^{m\times r}$ and $C\in S^{r\times n}$. We write $\rho A=r$. For example, if $S$ is a division ring, then $\rho A$ is just the rank of $A$. We say that $S$ is a {\bf Sylvester domain} \cite[Section 5.5]{Coh} if for any $P\in S^{\ell\times m}$ and $Q\in S^{m\times n}$ such that $PQ=0$, it follows that $\rho P+\rho Q\le m$.

We say that an element $s\in S$ {\bf right divides} $a\in S$ if $a=a's$ for some $a'\in S$. If $S$ is a domain and $a,b\in S\setminus\{0\}$, then $s$ is a {\bf highest common right factor} (HCRF) of $a$ and $b$ if $s$ right divides $a,b$ and every $s'\in S$ that right divides $a,b$ also right divides $s$. We say that $S$ is an {\bf HCRF domain} if every pair of nonzero elements in $S$ admits a HCRF. Special examples of HCRF domains are filtered rings satisfying the 2-term weak algorithm \cite[Section 2.8]{Coh} or more generally, 2-firs with right ACC$_1$  (ascending chain condition on principal right ideals) \cite[Exercise 3.2.1]{Coh}.

\begin{theorem}\label{t:sylv}
If $S$ is an HCRF domain and a Sylvester domain, then $S$ is an SN algebra.
\end{theorem}

\begin{proof}
Since $S$ is a Sylvester domain, it admits a universal skew field of fractions $D$ and this embedding preserves the inner rank by \cite[Theorem 7.5.13]{Coh}. Let $R$ be a central simple algebra and $\varphi:R\to R\otimes S$ a homomorphism. By Proposition \ref{p} there exists $c\in R\otimes S$ invertible in $R\otimes D$ such that $\varphi(x)=cxc^{-1}$ for all $x\in R$. Furthermore, if $\{r_1,\dots,r_d\}$ is a basis of $R$ and
$$c=\sum_l r_l\otimes t_l,\qquad c^{-1} =\sum_k r_k\otimes u_k$$
for $t_l\in S$ and $u_k\in D$, then $t_lu_k=s_{lk}\in S$ for all $1\le l,k\le d$ (here $u_k=s^{-1}s_k$ from Proposition \ref{p}). Since $S$ is an HCRF domain, we can assume that $t_1,\dots,t_d$ have no non-trivial common right factors (otherwise they have a nontrivial HCRF $e$ and we can replace $c$ with $c(1\otimes e^{-1})$). Fix $k$ such that $u_k\neq 0$. Then $(u_k,-1)^{\rm t}\in D^2$ belongs to the right kernel of the matrix
$$\begin{pmatrix}
t_1 & s_{1k} \\
\vdots & \vdots \\
t_d & s_{dk}
\end{pmatrix}\in S^{d\times 2}$$
which is therefore of (inner) rank $1$ over $D$. Since the embedding $S\subseteq D$ is inner rank preserving, this matrix is also of inner rank 1 over $S$, so
$$\begin{pmatrix}
t_1 & s_{1k} \\
\vdots & \vdots \\
t_d & s_{dk}
\end{pmatrix}=
\begin{pmatrix}
v_1 \\ \vdots \\ v_d
\end{pmatrix}
\begin{pmatrix}
w_1 & w_2
\end{pmatrix}$$
for some $v_i,w_j\in S$. Since $w_1$ right divides $t_l$ for every $l$ and $t_1,\dots,t_d$ are right coprime by assumption, we conclude that $w_1$ is invertible in $S$. By taking some $t_l\neq0$ we get
$$u_k=t_l^{-1}s_{lk}=w_1^{-1}v_l^{-1}v_lw_2=w_1^{-1}w_2\in S.$$
Consequently $c^{-1}\in R\otimes S$.
\end{proof}

\begin{corollary}\label{cor:free}
Every free algebra $F\langle X\rangle$ is an SN algebra.
\end{corollary}
\begin{proof} A free algebra is a filtered ring with a weak algorithm \cite[Theorem 2.5.3]{Coh}, so it is a HCRF domain and a fir (free ideal ring) by \cite[Theorem 2.4.6]{Coh} and hence a Sylvester domain by \cite[Proposition 5.5.1]{Coh}.
\end{proof}

Theorem \ref{t:sylv} has the following form for commutative rings.
\begin{corollary}\label{bez}
Every B\'ezout domain  is an SN algebra.
\end{corollary}
\begin{proof}
Every B\'ezout domain is a GCD domain, which is just a commutative HCRF domain. Moreover, by \cite[Proposition 2.3.17]{Coh} it is also a semifir and hence a Sylvester domain by \cite[Proposition 5.5.1]{Coh}. Therefore Theorem \ref{t:sylv} applies.
\end{proof}

In the next example we present a domain that is not an SN algebra;
cf.~\cite[Theorem 15]{RZ}.

\begin{example}\label{ex:elliptic}
Let $S=F[x,y]/(y^2-x^3-x)$. Then $S$ is a domain,
$$a=\begin{pmatrix}y & x \\ x^2 & y\end{pmatrix}\in M_2(S)$$
is invertible as a matrix over the field of fractions of $S$ and
$$a^{-1}=\begin{pmatrix}\frac{y}{x} & -1 \\ -x & \frac{y}{x}\end{pmatrix}.$$
Since every product of an entry in $a$ and an entry in $a^{-1}$ lies in $S$, it follows that
$$\varphi\colon M_2(F)\to M_2(F)\otimes S,\qquad u\mapsto aua^{-1}$$
is a well-defined homomorphism. Suppose there exists an invertible $c\in M_2(S)$ such that $\varphi(u)=cuc^{-1}$ for all $u\in M_2(F)$. Then $\gamma=\det(c)$ is invertible in $S$ and it is easy to see that this implies $\gamma\in F\setminus\{0\}$. Since $c^{-1}a$ commutes with every $u\in M_2(F)$ by the definition of $\varphi$, we have $c^{-1}a=f I_2$ for some $f\in S$. But then
$$\gamma f^2=\gamma \det(c^{-1}a)=\det(a)= x$$
contradicts the irreducibility of $x$ in $S$.
\end{example}

\section{Polynomial matrix algebras}\label{s:matrixpoly}

In this section we prove that matrix algebras over polynomial algebras are SN.

\begin{theorem}\label{t:matrixpoly}
	$M_n(F[\xi_1,\dots,\xi_s])$ is an SN algebra.
\end{theorem}
Besides the Quillen-Suslin theorem, saying that over $F[\xi_1,\dots,\xi_s]$ every finitely generated projective module is free, see \cite{Q,Sus}, the proof of this theorem is mostly based on the following simple factorization lemma.
\begin{lemma}
	\label{l:denominators}
	Let $A$ be a commutative algebra, which is a domain with field of fractions $K$, $R$ a central simple algebra and $a\in R\otimes M_n(A)$.
	Suppose that $a$ is invertible in $R\otimes M_n(K)$ and that $a(x\otimes 1)a^{-1}\in R\otimes M_n(A)$ for all $x\in R$. Then for $c\in M_n(A)$ the following are equivalent:
	\begin{enumerate}
		\item[(i)] There exists a factorization $a=u(1\otimes c)$ for some invertible $u\in R\otimes M_n(A)$;
		\item[(ii)] The left ideal
			\[
				I(a):=\{\, m\in M_n(A) \mid (1\otimes m)a^{-1}\in R\otimes M_n(A)\,\}\subseteq M_n(A)
			\]
			is generated by $c$;
		\item[(iii)] The rows of $c$ form a basis of the $A$-module
			\[
				M(a):=\{\, r\in A^{1\times n} \mid (1\otimes r)a^{-1}\in R\otimes A^{1\times n}\,\}.
			\]
	\end{enumerate}

\end{lemma}
\begin{proof}
	(i)$\Rightarrow$(ii): Let $u\in R\otimes M_n(A)$ be invertible such that $a=u(1\otimes c)$. Then for any $m\in M_n(A)$ we have that $(1\otimes m)a^{-1}=(1\otimes mc^{-1})u^{-1}$ lies in $R\otimes M_n(A)$ if and only if $mc^{-1}\in M_n(A)$, i.e., $m\in M_n(A)c$.

	(ii)$\Rightarrow$(i): 
	Let $\{r_1,\dots,r_d\}$ be a basis of $R$ and for fixed $i$ take $w_k,z_k\in R$ such that $\sum_k w_k r_j z_k=\delta_{ij}$, see \eqref{oh} at the beginning of Section \ref{s:prelim}. If $a=\sum_jr_j\otimes a_j$, then
	\[
		(1\otimes a_i)a^{-1}=\sum_k (w_k\otimes 1)a(z_k\otimes 1)a^{-1}\in R\otimes M_n(A).
	\]
	Since $i$ was arbitrary, this shows that all coefficients $a_i$ of $a$ lie in $I(a)$. Assuming that $I(a)=M_n(A)c$, we can thus factor $a=u(1\otimes c)$ for some $u\in R\otimes M_n(A)$. But then $u^{-1}=(1\otimes c)a^{-1}\in M_n(A)$, i.e., $u$ is invertible.

	(ii)$\Rightarrow$(iii): Suppose $I(a)=M_n(A)c$. Since $a$ is invertible over $K$, there exists a nonzero $e\in A$ such that $e \cdot 1\in I(a)$. Hence the rows of $c$ are clearly linearly independent. Given any $r\in M(a)$, we can extend $r$ by zero to form a matrix $m\in I(a)$ which has $r$ as one of its rows. By assumption $m\in M_n(A)c$. In particular, $r$ is a linear combination of the rows of $c$, which shows that they form a basis of $M(a)$.

	(iii)$\Rightarrow$(ii): Conversely, suppose that the rows of $c$ form a basis of $M(a)$. In particular, $c\in I(a)$. Moreover, for any $m\in I(a)$ the rows of $m$ lie in $M(a)$ and are, therefore, linear combinations of the rows of $c$, which implies that $m\in M_n(A)c$.
\end{proof}
\begin{proof}[Proof of Theorem~\ref{t:matrixpoly}]
	Let $A:= F[\xi_1,\dots,\xi_s]$, $K$ its field of fractions, $R$ a central simple algebra and $\varphi\colon R\to R\otimes M_n(A)$ a homomorphism. Since $M_n(K)$ is SN, there exists (after clearing denominators) $a\in R\otimes M_n(A)$, invertible in $R\otimes M_n(K)$, such that $\varphi(x)=a(x\otimes 1)a^{-1}$ for all $x\in R$.

	Fix any prime ideal $P$ of $A$. Then $M_n(A_P)/{\rm rad}(M_n(A_P))$ is canonically isomorphic to the simple algebra $M_n(F)\otimes A_P/PA_P$, see Lemma~\ref{lemrad}. Therefore, $M_n(A_P)$ is semilocal and by Theorem~\ref{thm:semiLoc} it is also an SN algebra.

	It follows that there exists an invertible $u_P\in R\otimes M_n(A_P)$ such that $\varphi(x)=u_P(x\otimes 1)u_P^{-1}$. Then $u_P^{-1}a$ commutes with all elements of $R\otimes 1$ and thus lies in $1\otimes M_n(A)$. This means, $a$ can be factored as $a=u_P(1\otimes c_P)$ for some $c_P\in M_n(A)$. By Lemma~\ref{l:denominators}, this implies that the $A_P$-module $A_PM(a)$ is free of rank $n$. Since the prime ideal $P$ was arbitrary, this shows that $M(a)$ is locally free of rank $n$. As being projective is a local property, this implies $M(a)$ is projective. By the Quillen-Suslin theorem $M(a)$ is free of rank $n$. We choose $c\in M_n(A)$ such that its rows form a basis of $M(a)$. Then again from Lemma~\ref{l:denominators}, we get a factorization $a=u(1\otimes c)$ where $u\in R\otimes M_n(A)$ is invertible. Now $\varphi(x)=u(x\otimes 1)u^{-1}$ for all $x\in R$.
\end{proof}

\section{Formal power series}\label{s:poly}

The aim of this section is to show that the property of being an SN algebra transfers from  $S$ to the formal power series algebra $S[[\xi]]$.

\begin{theorem}\label{pow}
$S$ is an SN algebra if and only if $S[[\xi]]$ is an SN algebra.
\end{theorem}

\begin{proof}
$(\Rightarrow)$ Let $R$ be a central simple algebra and let $\varphi:R\to R\otimes S[[\xi]]$ be a homomorphism. Since $R$ is finite-dimensional, we can identify $R\otimes S[[\xi]]$ with $(R\otimes S)[[\xi]]$ and write
$$\varphi(x)= \varphi_0(x) + \varphi_1(x)\xi+ \varphi_2(x)\xi^2 + \dots$$
where $\varphi_i:R\to R\otimes S$. Note that $\varphi_0$ is an algebra homomorphism. By assumption, there exists an invertible element $a\in R\otimes S$ such that 
$\varphi_0(x)=axa^{-1}$ for all $x\in R$. Considering the map $x\mapsto a^{-1}\varphi(x)a$ 
we see that without loss of generality we may assume that $\varphi_0(x)=x$ for all $x\in R$, so that 
\begin{equation}\label{ena1}\varphi(x)= x + \varphi_1(x)\xi+ \varphi_2(x)\xi^2 + \dots\end{equation}
Now apply Lemma \ref{l}. Thus, let $\{r_1,\dots,r_d\}$ be  a basis of $R$ and let $c_1,\dots,c_d\in R\otimes S[[\xi]]$ be such that \begin{equation}\label{ena2}
\sum_{k=1}^d c_kr_k =1\,\,\,\mbox{  and 
}\,\,\,\varphi(x)c_k = c_kx\end{equation} for all $x\in R$  and all  $k$. Writing $$c_k = \sum_{j=0}^\infty c_{kj}\xi^j,$$
where $c_{kj}\in R\otimes S$, it follows from \eqref{ena1} and  \eqref{ena2}  that 
\begin{equation}\label{ena3}xc_{k0} = c_{k0}x\end{equation}
 for all $x\in R$. Let us write $c_{k0} = \sum_j p_{kj}\otimes s_{kj}$
with the $s_{kj}$'s linearly independent. From \eqref{ena3} we infer that $$\sum_j (xp_{kj} - p_{kj}x)\otimes s_{kj} =0$$
for all $x\in R$, yielding $xp_{kj} - p_{kj}x=0$. Since $R$ is central this means that  each $p_{kj}$ is a scalar multiple of $1$. Accordingly, each $c_{k0}$ is of the form
$1\otimes t_k$ for some $t_k\in S$. From the first identity in $\eqref{ena2}$ one easily deduces that $\sum_{k=1}^d r_k \otimes t_k = 1\otimes 1$. Writing $1=\sum_{k=1}^d\lambda_k r_k$, where $\lambda_k\in F$, it follows that $t_k=\lambda_k 1$. We may assume that $\lambda_1\ne 0$. Accordingly, $c_{10}$ is a nonzero scalar multiple of unity of $R\otimes S$, implying that $c_1$ is invertible in $(R\otimes S)[[\xi]]$. Applying $\eqref{ena2}$ we arrive at $\varphi(x)= c_1xc_1^{-1}$ for all $x\in R$.

$(\Leftarrow)$ Straightforward; more generally, the SN property is clearly preserved by retractions. Here an algebra $S'$ is a retract of $S$ if $S'\subset S$ and there exists a homomorphism $\pi:S\to S'$ that restricts to the identity map on $S'$.
\end{proof}

\end{document}